\documentclass[12pt]{article}
\usepackage{indentfirst,latexsym,bm}
\usepackage{amsmath}
\usepackage{amssymb}
\usepackage{amsfonts}
\usepackage{amsthm}
\usepackage{mathrsfs}
\usepackage{cite}
\usepackage[all]{xy}
\usepackage{newlfont}
\usepackage[noblocks]{authblk}

\newtheorem{prop}{Proposition}[section]
\newtheorem{thm}{Theorem}[section]
\newtheorem{lem}{Lemma}[section]
\newtheorem{cor}[thm]{Corollary}
\newtheorem{rek}[lem]{Remark}

\ifx\pdfoutput\undefined 
\usepackage[dvipdfm,bookmarks,unicode]{hyperref}
\else
\usepackage[pdftex,bookmarks,unicode]{hyperref}
\fi

\begin{document}

\title{Existence and multiplicity of solutions for nonlocal systems with Kirch\/hof\/f type}

\author[a]{Zhitao Zhang\thanks{Supported by NSFC 11325107,11271353, 11331010.}}
\author[b]{Yimin Sun\thanks{Corresponding author.}}

\affil[a]{\small Academy of Mathematics and Systems Science,Chinese Academy of Sciences,\authorcr Beijing 100190, China\authorcr zzt@math.ac.cn}

\affil[b]{\small Department of Mathematics, Northwest University, Xi'an 710127, China\authorcr ymsun@nwu.edu.cn}        

\date{}

\maketitle

\renewcommand{\abstractname}{Abstract}

\begin{abstract}
Firstly, we use Nehari manifold and Mountain Pass Lemma to prove an
existence result of positive solutions for a class of nonlocal
elliptic  system with Kirchhoff type. Then a multiplicity result is
established by  cohomological index of Fadell and Rabinowitz. We
also consider the critical case and prove existence of positive
least energy solution when the parameter $\beta$ is sufficiently large.

\end{abstract}

\emph{Keywords:} Critical point; Nehari manifold; cohomological
index.

\emph{AMS Subjection Classification(2000):} 35J55; 35J65; 49J40.

\section{Introduction}
 In this paper, we are concerned with existence and multiplicity results for the following
 nonlocal boundary value problem of Kirch\/hof\/f type
\begin{equation}\label{1}
\begin{cases}
 -(a_1+b_1\int_{\Omega} |\nabla u_1|^2\,\mathrm{d}x)\Delta u_1\ =
  \lambda_1 |u_1|^{p-1}u_1+ \beta |u_1|^{\tfrac{q-3}{2}} |u_2|^{\tfrac{q+1}{2}}u_1 \qquad in \ \Omega,\\
-(a_2+b_2\int_{\Omega} |\nabla u_2|^2\,\mathrm{d}x)\Delta u_2\ =
 \lambda_2 |u_2|^{p-1}u_2+ \beta |u_1|^{\tfrac{q+1}{2}} |u_2|^{\tfrac{q-3}{2}}u_2 \qquad in \ \Omega,\\
u_1=u_2=0 \quad on\ \partial\Omega,
\end{cases}
\end{equation}
where $\Omega\subset{\mathbb{R}^N},$ for $N=1,2,3,$ is a bounded
smooth domain, $\beta\in\mathbb{R}$ and $a_i, b_i,
{\lambda}_i$ are positive constants respectively for  $i=1,2$.
Moreover, $p$ and $q$ are two positive numbers that satisfy some
conditions to be stated later on.

Problem (\ref{1}) is called nonlocal because of the presence of the terms $b_i\int_{\Omega}|\nabla u|^2\,\mathrm{d}x$, $i=1,2.$
And the operator $b(\int_{\Omega}|\nabla u|^2\,\mathrm{d}x)\Delta u$ appears in the Kirchhoff equation
\begin{equation}\label{2}
 \begin{cases}
 -(a+b\int_{\Omega} |\nabla u|^2\,\mathrm{d}x)\Delta u\ = f(x,u) \qquad in \ \Omega\\
u=0 \qquad on\ \quad\partial\Omega,
\end{cases}
\end{equation} related to  the stationary analogue of the equation
\begin{equation*}
    u_{t\/t}-(a+ b \int_{\Omega}|\nabla u|^2\,\mathrm{d}x)\Delta u= f(x,t),
\end{equation*}
which was proposed by Kirch\/hof\/f\cite{Kirchhoff1883} as an extension of the classical
D'Alembert's wave equation for free vibrations of elastic strings.
Equation (\ref{2}) has attracted a considerable attention only after Lions\cite{Lions1978}
 presented an abstract framework to this problem.
 Some interesting and further results can be found in
\cite{Alves-Correa-Fig2010,Biagio2010,Ma-Munoz2003,Mao-Zhang,Perera-Zhang06}
and the references therein, among which Perera and Zhang obtained nontrivial solutions of  a
class of nonlocal quasi-linear elliptic  boundary value problems using the Yang index, invariant sets of descent
flow and critical groups in \cite{Perera-Zhang06,Zhang-Perera06}.

Such nonlocal problems also model
several physical and biological systems where $u$ describes a
process which depends on the average of itself, for example the
population density (see, e.g.
\cite{Alves-Correa-Ma05,Andrade-Ma97,Alves-Correa2001,
Vasconcellos92,Chipot-Lovat97,Chipot-Rodrigues92}).

Without such nonlocal terms, (\ref{1}) is related to  the following nonlinear  Schr{\"o}dinger system
\begin{equation}\label{BECS}
 \begin{cases}
 -\Delta u_1 + \lambda_1 u_1= \mu_1 u_1^3+\beta u_1 u_2^2 \qquad in \ \Omega,\\
 -\Delta u_2 + \lambda_2 u_2= \mu_2 u_2^3 + \beta u_2 u_1^2 \qquad in \ \Omega,
 \end{cases}
\end{equation}
stemming from many physical problems, especially in nonlinear optics and in the Hartree-Fock theory for Bose-Einstein condensates;
see, for example, \cite{AC07,BWW07,DanWei09,WW08,FL08,LW05,LW06,Sirakov07}.


In the special case $f(x,u)=\lambda(u^+)^p$ with $3<p<2^{*}-1$,
problem (\ref{2}) admits a positive solution $w(x)=w(x;a,b,\lambda)$
through the following minimization
\begin{equation*}
    \inf_{u \in \tilde{\mathcal{N}}}I(u)
\end{equation*}
where $$ I(u)= {\frac{a}{2}}\int_{\Omega} |\nabla u|^2\,\mathrm{d}x
+ \frac{b}{4}\left(\int_{\Omega} |\nabla u|^2\,\mathrm{d}x\right)^2-
\frac{\lambda}{p+1}\int_{\Omega}(u^+)^{p+1}\,\mathrm{d}x$$ and
$$\tilde{\mathcal{N}}=\left\{u\in H_0^1(\Omega)\setminus\{0\}: a\int_{\Omega}
|\nabla u|^2\,\mathrm{d}x + b\left(\int_{\Omega} |\nabla
u|^2\,\mathrm{d}x\right)^2= \lambda\int_{\Omega}(u^+)^{p+1}\,\mathrm{d}x
\right\}.$$

If we set $U_1(x)=w(x;a_1,b_1,\lambda_1)$ and
$U_2(x)=w(x;a_2,b_2,\lambda_2)$, then problem (\ref{1}) admits two
semi-positive  solutions $\mathbf{u}_1=(U_1,0)$  and
$\mathbf{u}_2=(0,U_2)$, for all $\beta \in \mathbb{R}.$

We are interested in solutions $\mathbf{u}=(u_1,u_2)$ of problem
(\ref{1}) with all components $u_j>0$, $j=1,2$. These are called
positive solutions as opposed to semi-positive solutions.

 Firstly, we study problem $(\ref{1})$ with the  following subcritical growth:
\begin{equation}\label{subcritical case}
     p=q \in (3,2^*-1)\quad with \quad 2^*=
  \left \{ \begin{array}{cc}
    6  & for\, N=3,\\
    \infty  & for\, N=1,2.
\end{array}\right.
\end{equation}

For $\beta>-\sqrt{\lambda_1\lambda_2}$, the Nehari manifold
corresponding
 to the energy functional of problem (\ref{1}) can be well defined (see the proof of \emph{Lemma \ref{2.1}(i)}).
 We will show that the Morse index of semi-positive solution
 $\mathbf{u}_i$ $(i=1,2)$ is exactly one under above assumptions (see Lemma \ref{strict local min}). This fact, jointly with the Mountain Pass Lemma, yields our first result of this paper.
\begin{thm}\label{thm1}
Let $p=q\in(3,2^*-1)$ and $\beta>-\sqrt{\lambda_1\lambda_2}$. Assume
that $N=1$ or $\Omega$ is radially symmetric. Problem (\ref{1})
possesses a positive solution $\mathbf{u}^*$ with Morse index at
least two.
\end{thm}
\begin{rek}\label{rem 1.1}
The condition that $N=1$ or $\Omega$ is radially symmetric is to
ensure the uniqueness of positive solution for problem (\ref{2})
with $f(x,u)=\lambda(u^+)^p$ (see\cite{Anello2010}).
\end{rek}

For $\beta\le -\sqrt{\lambda_1\lambda_2}$, we can get infinitely
many positive  solutions of problem (\ref{1}) under some symmetric
conditions. In fact, since the oddness of nonlinearities and interaction
terms, problem (\ref{1}) satisfies the hypotheses of the Fountain
Theorem (\cite{Willem}) which immediately yields  a sequence of
solutions. However, these solutions may be sign-changing. In order
to obtain positive solutions, we consider the fully symmetric case
\begin{equation}\label{fully sym}
    a:=a_1=a_2,\quad b:=b_1=b_2 \quad and \quad \lambda:=\lambda_1=\lambda_2.
\end{equation}
By (\ref{fully sym}) and $\beta\le-\sqrt{\lambda_1\lambda_2}$,
problem (\ref{1}) has no positive solutions whose two components are
the same. Furthermore, a free $\mathbb{Z}_2-$space can be defined by
a Nehari manifold.
 Using cohomological index of Fadell and Rabinowitz
\cite{Fadell-Rabinowitz1978} for the $\mathbb{Z}_2-$free actions,
we obtain the following multiplicity result.
\begin{thm}\label{thm2}
Assume that condition (\ref{fully sym}) holds. Let $p=q\in(3,
2^*-1)$ and $\beta\le-\lambda$. Then problem (\ref{1}) admits a
sequence of positive solutions $\{\mathbf{u}_k\}$ with
$\|\mathbf{u}_k\|_{\mathcal{L}^{\infty}(\Omega)}\to \infty.$
\end{thm}

 Next, we consider the following critical growth case:
\begin{equation}\label{critical case}
p=5  \quad and\,\quad  q\in(3,5)\,\quad with\,\quad N=3.
\end{equation}
Since the dimension is $N=3$, $6=\tfrac{2N}{N-2}$ is the critical
Sobolev exponent. Hence there are critical nonlinearities and
coupling interaction terms in this nonlocal system. To the authors'
knowledge, however, there are few results on nonlocal systems with
critical nonlinearities. In the case of a single nonlocal
Kirchhoff-type equation with critical nonlinearity in $\mathbb{R}^3$
\begin{equation*}
    \begin{cases}
-[M(\int_{\Omega} |\nabla u|^2\,\mathrm{d}x)]\Delta u=
 \lambda f(x,u)+ u^5 \quad in \ \Omega,\\
u>0  \quad in \; \Omega,\qquad u=0 \quad on\ \partial\Omega,
\end{cases}
\end{equation*}
Alves-Correa-Figueiredo \cite{Alves-Correa-Fig2010} obtained the
existence of positive solutions for sufficiently large $\lambda>0$
($f$ is subcritical growth). They used the concentration compactness
principle to prove that the compactness is recovered when
$\lambda>0$ is sufficiently large. Inspired by this, we shall prove
the following existence result for problem (\ref{1}).
\begin{thm}\label{thm3}
If condition (\ref{critical case}) holds, then problem (\ref{1})
has a positive least energy solution for sufficiently large
$\beta>0$.
\end{thm}

In fact, the conclusion of {\emph{Theorem \ref{thm3}}} is also valid
for subcritical growth case (i.e., condition (\ref{subcritical
case})). This jointly with {\emph{Theorem \ref{thm1}}} yields that

\begin{cor}\label{cor 1.4}
Under the assumptions of Theorem \ref{thm1}, problem (\ref{1})
admits at least two positive solutions for sufficiently large
$\beta>0$, one of which has Morse index one and the other at least two.
\end{cor}

This paper is organized as follows:
\begin{itemize}
  \item Section 2: proof  of Theorem \ref{thm1};
  \item Section 3: proof of Theorem \ref{thm2};
  \item Section 4: proofs of Theorem \ref{thm3} and Corollary \ref{cor 1.4}.
\end{itemize}

Notation

\begin{itemize}
  \item $H=H_0^1(\Omega)$ endowed with scalar product and norm

  $\langle u,v\rangle=\int_{\Omega}\nabla u\cdot\nabla v\;\mathrm{d}\,x, \qquad \|u\|^2=\langle
  u,u\rangle;$
  \item $\mathcal{H}=H\times H$ whose elements will be denoted by
  $\mathbf{u}=(u_1,u_2)$; its norm is

 $ \|\mathbf{u}\|^2=\|u_1\|^2+\|u_2\|^2;$

 \item $\mathcal{L}^p(\Omega)=L^p(\Omega)\times L^p(\Omega)$ with its norm $|\mathbf{u}|_p^p=|u_1|^p_p+|u_2|^p_p$
 for any $p\in(1,\infty)$;

  \item $$S\stackrel{\mathrm{def}}{=} \inf_{u\in H\setminus\{0\} \atop |u|_6=1}\int_{\Omega}|\nabla
  u|^2\quad\mathrm{d}x.$$
\end{itemize}

\section{The proof of Theorem \ref{thm1}}
We assume that $p=q\in(3,2^*-1)$ and
$\beta>-\sqrt{\lambda_1\lambda_2}$ in the whole section.

For $\mathbf{u}:=(u_1,u_2)\in\mathcal{H}$, we set
\begin{eqnarray*}
  \Phi(\mathbf{u})&=& \frac{1}{2}(a_1\|u_1\|^2+a_2\|u_2\|^2)+
    \frac{1}{4}(b_1\|u_1\|^4+b_2\|u_2\|^4) \\
   && -\frac{1}{p+1}\int_{\Omega}(\lambda_1|u_1|^{p+1}+\lambda_2|u_2|^{p+1}
    +2\beta|u_1|^{\frac{p+1}{2}}|u_2|^{\frac{p+1}{2}})\,\mathrm{d}x.
\end{eqnarray*}
Then $\Phi\in\mathcal{C}^2(\mathcal{H},\mathbb{R})$ by Sobolev
embedding theorem. We have that critical points of $\Phi$ are
solutions of problem (\ref{1}).

Next, we introduce the Nehari manifold
\begin{equation*}
    \mathcal{N}=\{\;\mathbf{u}\in\mathcal{H}\setminus\{\mathbf{0}\}:
    \quad F(\mathbf{u})\stackrel{\mathrm{def}}{=}\langle\Phi^{'}(\mathbf{u}),\mathbf{u}\rangle=0\}.
\end{equation*}
Clearly, all nontrivial critical points  of $\Phi$ are contained in
$\mathcal{N}$.
\begin{lem}\label{2.1}
$(i)$ $\mathcal{N}$ is homeomorphic to the unit sphere of
$\mathcal{H}$, and  there exists $\rho>0$  such that
$\|\mathbf{u}\|\ge \rho$, \;$\forall \mathbf{u}\in\mathcal{N}$.

$(ii)$ $\mathcal{N}$ is a  $\mathcal{C}^1$ complete manifold of
codimension one in $\mathcal{H}$.

$(iii)$ If $\mathbf{u}$ is a critical point of the restriction
$\Phi_{\mathcal{N}}$ of $\Phi$ to $\mathcal{N}$, then $\mathbf{u}$
is a nontrivial critical point of $\Phi$.

$(iv)$
$\Phi(\mathbf{u})=\frac{p-1}{2(p+1)}(a_1\|u_1\|^2+a_2\|u_2\|^2)$ $+
$$\frac{p-3}{4(p+1)}(b_1\|u_1\|^4+b_2\|u_2\|^4)$
, for all $\mathbf{u}\in \mathcal{N}$.

 $(v)$ \label{PS}$\Phi_{\mathcal{N}}$ satisfies Palais-Smale condition.
\end{lem}
\begin{proof}
(i) For any $\mathbf{u}\in\mathcal{H}\setminus\{\mathbf{0}\}$, one
has that
\begin{eqnarray}\label{the exist of t}
  \nonumber t\mathbf{u}\in \mathcal{N}&\Longleftrightarrow& a_1\|u_1\|^2+a_2\|u_2\|^2+
    t^2(b_1\|u_1\|^4+b_2\|u_2\|^4) \\
   && =t^{p-1}\int_{\Omega}(\lambda_1|u_1|^{p+1}+\lambda_2|u_2|^{p+1}
    +2\beta|u_1|^{\frac{p+1}{2}}|u_2|^{\frac{p+1}{2}})\,\mathrm{d}x.
\end{eqnarray}

Notice that for any $\beta>-\sqrt{\lambda_1\lambda_2}$ and
$\mathbf{u}\in\mathcal{H}\setminus\{\mathbf{0}\}$
\begin{equation*}
    \int_{\Omega}(\lambda_1|u_1|^{p+1}+\lambda_2|u_2|^{p+1}
    +2\beta|u_1|^{\frac{p+1}{2}}|u_2|^{\frac{p+1}{2}})\,\mathrm{d}x>0.
\end{equation*}

 We claim that for any $\mathbf{u}\in\mathcal{H}\setminus\{\mathbf{0}\}$
there exists a unique $t(\mathbf{u})\in\mathbb{R}^{+}$ such that
$t(\mathbf{u})\mathbf{u}\in\mathcal{N}$.

 Indeed, for fixed $\mathbf{u}\in\mathcal{H}\setminus\{\mathbf{0}\}$, $(\ref{the exist of t})$ implies that $t\mathbf{u}\in\mathcal{N}$ if and only if $t$  is a positive zero of the  following function
 \begin{equation}\label{the fuc of f}
    f(t)=At^{p-1}-Bt^{2}-C,
\end{equation}
where $A,B,C$ are positive constants $($depending on $\mathbf{u}$$)$. Notice that
$f(0)<0,\,f^{'}(0)=0,\;  f^{''}(0)<0,\; f(+\infty)=+\infty\quad $and
 there exists a unique $t_1>0$ such that $f^{'}(t_1)=0$.
Thus, there exists a unique positive number $t^{*}$ such
 that $f(t^{*})=0$. Hence, there exists a unique $t(\mathbf{u})\in\mathbb{R}^{+}$
 such that $t(\mathbf{u})\mathbf{u}\in\mathcal{M}.$

 In order to prove the continuity of $t(\mathbf{u})$, we assume that
 $\mathbf{u}_n\to \mathbf{u}_0$ in
 $\mathcal{H}\setminus\{\mathbf{0}\}$.
 It follows from (\ref{the exist of t}) that $\{t(\mathbf{u}_n)\}$ is
 bounded. Passing if necessary to a subsequence, we can assume that
 $t(\mathbf{u}_n)\to t_0,$ then $t_{0}=t(\mathbf{u}_0)$
 by $(\ref{the exist of t})$ and the uniqueness of $t(\mathbf{u}_0)$.
 Hence $t(\mathbf{u}_n)\to t(\mathbf{u}_0)$. Moreover, the
 inverse map of $t(\mathbf{u})|_{S^{\infty}}: S^{\infty}\to \mathcal{N}$
can be defined by
\begin{equation*}
   \mathbf{u}\to \,\mathbf{u}/\|\mathbf{u}\|,
\end{equation*}
which is also continuous. Therefore, the Nehari manifold
$\mathcal{N}$ is homeomorphic to the unit sphere of $\mathcal{H}$.

Moreover, if $\mathbf{u}=(u_1,u_2)\in\mathcal{N}$ then by H{\"o}lder
inequality and Sobolev embedding theorem, we get that
\begin{eqnarray*}
   &&  a_1\|u_1\|^2+a_2\|u_2\|^2+b_1\|u_1\|^4+b_2\|u_2\|^4 \\
   &=&  \int_{\Omega}(\lambda_1|u_1|^{p+1}+\lambda_2|u_2|^{p+1}
    +2\beta|u_1|^{\frac{p+1}{2}}|u_2|^{\frac{p+1}{2}})\,\mathrm{d}x\\
  &\le& c_1\|u_1\|^{p+1}+c_2\|u_2\|^{p+1}.
\end{eqnarray*}
It follows from $p>3$ that there exists $\rho>0$ such that
$\|\mathbf{u}\|\ge \rho$ for any
 $\mathbf{u}\in\mathcal{N}$.

(ii) Notice that for any $\mathbf{u}\in\mathcal{H}\setminus\{\mathbf{0}\}$
\begin{eqnarray*}
   \langle F^{'}(\mathbf{u}),\mathbf{u}\rangle&=&2(a_1\|u_1\|^2+a_2\|u_2\|^2)+
    4(b_1\|u_1\|^4+b_2\|u_2\|^4) \\
   &&-(p+1)\int_{\Omega}(\lambda_1|u_1|^{p+1}+\lambda_2|u_2|^{p+1}
    +2\beta|u_1|^{\frac{p+1}{2}}|u_2|^{\frac{p+1}{2}})\,\mathrm{d}x,
\end{eqnarray*}
and
\begin{eqnarray}\label{equiv of N}
            \nonumber \mathbf{u}\in\mathcal{N} &\Leftrightarrow& a_1\|u_1\|^2+a_2\|u_2\|^2+b_1\|u_1\|^4+b_2\|u_2\|^4 \\
             && =\int_{\Omega}(\lambda_1|u_1|^{p+1}+\lambda_2|u_2|^{p+1}
    +2\beta|u_1|^{\frac{p+1}{2}}|u_2|^{\frac{p+1}{2}})\,\mathrm{d}x.
           \end{eqnarray}
It follows that
\begin{eqnarray}\label{negative definite of F}
  \nonumber \langle F^{'}(\mathbf{u}),\mathbf{u}\rangle&=& (1-p)(a_1\|u_1\|^2+a_2\|u_2\|^2)+(3-p)(b_1\|u_1\|^4+b_2\|u_2\|^4) \\
  \nonumber &<& -2\,\rho^2\,\min\{a_1,a_2\} \\
   &<& 0,\quad \forall  \;\mathbf{u}\in\mathcal{N}.
\end{eqnarray}
This jointly with (i) and Implicit Function Theorem yields that
$\mathcal{N}$ is a $\mathcal{C}^1$ complete manifold of codimension
one in $\mathcal{H}$.

(iii) If $\mathbf{u}$ is a critical point of $\Phi_{\mathcal{N}}$,
then there exists $\omega(\mathbf{u})\in\mathbb{R}$ such that
$\Phi^{'}(\mathbf{u})=\omega(\mathbf{u})F^{'}(\mathbf{u})$. It
follows from the fact $\mathbf{u}\in\mathcal{N}$ that
\begin{equation*}
    \omega(\mathbf{u})\langle
F^{'}(\mathbf{u}),\mathbf{u}\rangle=\langle\Phi^{'}(\mathbf{u}),\mathbf{u}\rangle=F(\mathbf{u})=0.
\end{equation*}
Thus we deduced from $(\ref{negative definite of F})$ that $\omega(\mathbf{u})=0$ and
$\Phi^{'}(\mathbf{u})=\mathbf{0}$.

(iv) This results from a simple calculation by
$(\ref{equiv of N})$.

(v) Assume that $\{\mathbf{u}_n\} \subset\mathcal{N}$ is a $(PS)_c$ sequence
of $\Phi_{\mathcal{N}}$, that is,
\begin{equation*}
\Phi(\mathbf{u}_n)\to c  \quad and
 \quad \nabla_{\mathcal{N}}\Phi(\mathbf{u}_n)\to \mathbf{0}.
\end{equation*}
Then there exist $\omega_n\in\mathbb{R}$ such that
\begin{equation}\label{the derivative of Phi}
    \nabla_{\mathcal{N}}\Phi(\mathbf{u}_n)=\Phi^{'}(\mathbf{u}_n)-\omega_n
F^{'}(\mathbf{u}_n).
\end{equation}
By  $(iv)$ and $\Phi(\mathbf{u}_n)\to c$ we have that
$\|\mathbf{u}_n\|\le C<+\infty$.
Hence,
\begin{equation*}
  \omega_n\langle F^{'}(\mathbf{u}_n),\mathbf{u}_n\rangle=\langle \Phi^{'}(\mathbf{u}_n),\mathbf{u}_n\rangle
-\langle \nabla\Phi(\mathbf{u}_n),\mathbf{u}_n\rangle=o(1).
\end{equation*}
It follows from $(\ref{negative definite of F})$ that
\begin{equation}\label{omega to 0}
    \omega_n\to 0.
\end{equation}

Notice that for all $\mathbf{v}=(v_1,v_2)\in \mathcal{H}$
\begin{eqnarray*}
  \langle F^{'}(\mathbf{u}_n),\mathbf{v}\rangle &=& (2 a_1+4 b_1\|u_{n,1}\|^2) \int_{\Omega}\nabla u_{n,1}\cdot\nabla v_1\,\mathrm{d}x+(2 a_2+4 b_2\|u_{n,2}\|^2)\int_{\Omega}\nabla u_{n,2}\cdot \nabla v_2\,\mathrm{d}x\\
  && -(p+1)\int_{\Omega}(\lambda_1|u_{n,1}|^{p-1}u_{n,1}v_1+\lambda_2|u_{n,2}|^{p-1}u_{n,2}v_2)\,\mathrm{d}x \\
   &&-(p+1)\beta\int_{\Omega}(|u_{n,1}|^{\frac{p-3}{2}}|u_{n,2}|^{\frac{p+1}{2}}u_{n,1}v_1+|u_{n,2}|^{\frac{p-3}{2}}|u_{n,1}|^{\frac{p+1}{2}}u_{n,2}v_2)\,\mathrm{d}x.
\end{eqnarray*}

By H{\"o}lder inequality and Sobolev embedding theorem, we have
\begin{eqnarray*}
  \left|\int_{\Omega}|u_{n,1}|^{p-1}u_{n,1}v_1\,\mathrm{d}x\right| &\le& \int_{\Omega}|u_{n,1}|^{p}|v_1|\,\mathrm{d}x \\
  &\le& \left(\int_{\Omega}|u_{n,1}|^{p+1}\,\mathrm{d}x\right)^{\frac{p}{p+1}}\left(\int_{\Omega}|v_1|^{p+1}\,\mathrm{d}x\right)^{\frac{1}{p+1}}\\
   &\le&C \|u_{n,1}\|^{p}\|v_1\|\\
   &\le&C\|\mathbf{u}_n\|^p\|\mathbf{v}\| ,
   \end{eqnarray*}
   and
   \begin{eqnarray*}
    &&\left|\int_{\Omega}|u_{n,1}|^{\frac{p-3}{2}}|u_{n,2}|^{\frac{p+1}{2}}u_{n,1}v_1\,\mathrm{d}x\right|\\ &\le& \int_{\Omega}|u_{n,1}|^{\frac{p-1}{2}}|u_{n,2}|^{\frac{p+1}{2}}|v_1|\,\mathrm{d}x\\
      &\le& \left(\int_{\Omega}|u_{n,1}|^{p+1}\,\mathrm{d}x\right)^{\frac{p-1}{2(p+1)}}\left(\int_{\Omega}|u_{n,2}|^{p+1}\,\mathrm{d}x\right)^{\frac{1}{2}}\left(\int_{\Omega}|v_1|^{p+1}\,\mathrm{d}x\right)^{\frac{1}{p+1}} \\
      &=& \left|u_{n,1}\right|^{\frac{p-1}{2}}_{p+1}\left|u_{n,2}\right|^{\frac{p+1}{2}}_{p+1}\left|v_1\right|_{p+1} \\
      &\le& C\|u_{n,1}\|^{\frac{p-1}{2}}\|u_{n,2}\|^{\frac{p+1}{2}}\|v_1\| \\
    &=& C\|\mathbf{u}_n\|^p\|\mathbf{v}\|.
   \end{eqnarray*}
Similarly, we get estimates for
\begin{equation*}
   \left|\int_{\Omega}|u_{n,2}|^{p-1}u_{n,2}v_2\,\mathrm{d}x\right|
   \qquad  and \quad
   \left|\int_{\Omega}|u_{n,2}|^{\frac{p-3}{2}}|u_{n,1}|^{\frac{p+1}{2}}u_{n,2}v_2\,\mathrm{d}x\right|.
\end{equation*}

It follows that
$$\|F^{'}(\mathbf{u}_n)\|\le C(\|\mathbf{u}_n\|+\|\mathbf{u}_n\|^3+\|\mathbf{u}_n\|^p),$$
and hence, $\|F^{'}(\mathbf{u}_n)\|\le C<+\infty$ since $\mathbf{u}_n$ is bounded.
 This jointly with $(\ref{the derivative of Phi})$ and $(\ref{omega to 0})$
yields that $\Phi^{'}(\mathbf{u}_n)\to\mathbf{0}$, that is, $\{\mathbf{u}_n\}$ is a $(PS)_c$ sequence of $\Phi$.

Notice that ${\mathbf{u}_n}$ is bounded in $\mathcal{H}$. Passing if necessary to a
subsequence, we can assume that $\mathbf{u}_n\rightharpoonup
\mathbf{u}$ in $\mathcal{H}$ and $\mathbf{u}_n \to \mathbf{u}$ in
$\mathcal{L}^{p+1}(\Omega)$. So
\begin{eqnarray}\label{PScondition}
\nonumber&&(a_1+b_1\|u_{n,1}\|^2)\int_{\Omega}\nabla
u_{n,1}\cdot\nabla(u_{n,1}-u_1)\mathrm{d}\,x
+(a_2+b_2\|u_{n,2}\|^2)\int_{\Omega}\nabla u_{n,2}\cdot\nabla(u_{n,2}-u_2)\mathrm{d}\,x\\
   &=&\langle \Phi^{'}(\mathbf{u}_n),\mathbf{u}_n-\mathbf{u}\rangle
+\int_{\Omega}[\lambda_1 |u_{n,1}|^{p-1}u_{n,1}(u_{n,1}-u_1)+\lambda_2
|u_{n,2}|^{p-1}u_{n,2}(u_{n,2}-u_2)]\;\mathrm{d}\,x\nonumber \\
&&+\beta\int_{\Omega}[|u_{n,1}|^{\frac{p-3}{2}}|u_{n,2}|^{\frac{p+1}{2}}u_{n,1}(u_{n,1}-u_1)+
|u_{n,1}|^{\frac{p+1}{2}}|u_{n,2}|^{\frac{p-3}{2}}u_{n,2}(u_{n,2}-u_2)]\;\mathrm{d}x\nonumber\\
   &\to&0.
\end{eqnarray}
Notice that
\begin{eqnarray*}
  \lim_{n\to\infty}\int_{\Omega}\nabla u_{n,1}\cdot \nabla u_1\quad\mathrm{d}x &=& \int_{\Omega}|\nabla u_1|^2\quad\mathrm{d}x \\
   &\le&  \liminf_{n\to\infty}\int_{\Omega}|\nabla u_{n,1}|^2\quad\mathrm{d}x\\
   &=&\lim_{n\to\infty}\int_{\Omega}|\nabla u_{n,1}|^2\quad\mathrm{d}x,
\end{eqnarray*} by Fatou's Lemma, that is,$$\lim_{n\to\infty}\int_{\Omega}\nabla u_{n,1}\cdot\nabla(u_{n,1}-u_1)\;\mathrm{d}\,x \ge 0.$$
Similarly, we also have that
  $$\lim_{n\to\infty}\int_{\Omega}\nabla u_{n,2}\cdot\nabla(u_{n,2}-u_2)\;\mathrm{d}\,x\ge 0.$$
It follows from  (\ref{PScondition}) that
\begin{eqnarray*}
     (a_1+b_1\|u_{n,1}\|^2)\int_{\Omega}\nabla u_{n,1}\cdot\nabla(u_{n,1}-u_1)\;\mathrm{d}\,x
 &\to& 0,\\
   (a_2+b_2\|u_{n,2}\|^2)\int_{\Omega}\nabla u_{n,2}\cdot\nabla(u_{n,2}-u_2)\;\mathrm{d}\,x&\to& 0.
\end{eqnarray*}
This jointly with the fact that $\|u_{n,i}\|$ is bounded, yields
that $\|u_{n,i}\|\to \|u_{i}\|$, then $u_{n,i}\to u_i$ in $H$ for
$i=1,2$, that is, $\mathbf{u}_n\to \mathbf{u}$ in $\mathcal{H}$.
Moreover, we have that $\mathbf{u}\in\mathcal{N}$  by
$\mathbf{u}_n\in\mathcal{N}$ and
$\langle\Phi^{'}(\mathbf{u}_n),\mathbf{u}_n\rangle\to 0$.
\end{proof}

\begin{rek}
{\emph Lemma \ref{2.1}} implies that
$\inf_{\mathbf{u}\in\mathcal{N}}\Phi(\mathbf{u})$ can be achieved,
giving rise to a nontrivial solution of problem (\ref{1}). One may
suspect that such a solution is a semi-positive one referred in the
Introduction. However, the proof of  {\emph Corollary \ref{cor 1.4}}
show that this can not occur when $\beta>0$ is large enough.
\end{rek}

In order to demonstrate that there exists a nontrivial solution of
problem (\ref{1}) different from semi-positive ones
$\mathbf{u}_1=(U_1,0)$ and $\mathbf{u}_2=(0,U_2)$, we need the
following lemma, which immediately yields that the Morse index of
$\mathbf{u}_j$ is exactly one for $j=1,2$.
\begin{lem}\label{strict local min}
$\mathbf{u}_j$, $j=1,2$, are strict local minima of $\Phi$ on
$\mathcal{N}$.
\end{lem}
\begin{proof}
For $j=1,2,$ as referred in the Introduction, $U_j$ achieves
$$\inf_{u\in\mathcal{N}_j}I_j(u)$$ with
\begin{eqnarray*}
      I_j(u) &= &\frac{a_j}{2}\| u\|^2 + \frac{b_j}{4}\|u\|^4-
\frac{\lambda_j}{p+1}|u^+|^{p+1}_{p+1}\qquad \quad and \\
      \mathcal{N}_j &=&\{u\in H\setminus\{0\}:a_j\| u\|^2 +
b_j\| u\|^4=\lambda_j|u^+|^{p+1}_{p+1}\}.
      \end{eqnarray*}
Hence, by $\Phi^{'}(\mathbf{u}_j)=0$ we get that
\begin{equation*}
    D^2_{\mathcal{N}_j}I_j(U_j)[h]^2=I^{''}_j(U_j)[h]^2\ge c_j \|h\|^2,\qquad\forall h\in T_{U_j}\mathcal{N}_j,
\end{equation*}
where $D^2_{\mathcal{N}_j}I_j(U_j)$ denotes the second derivative of $\Phi$ constrained on $\mathcal{N}$.
 For any $\mathbf{h}=(h_1,h_2)\in
T_{\mathbf{u}_1}\mathcal{N}$, we have that $h_1\in
T_{U_1}\mathcal{N}_1$ by a simple calculation. Therefore
\begin{eqnarray*}
  D^2_{\mathcal{N}}\Phi(\mathbf{u}_1)[\mathbf{h}]^2 &=&\Phi^{''}(\mathbf{u}_1)[\mathbf{h}]^2 \\
   &=& I^{''}(U_1)[h_1]^2 +a_2\|h_2\|^2\\
   &\ge& c_1\|h_1\|^2+a_2\|h_2\|^2,
\end{eqnarray*}
which yields that $\mathbf{u}_1$ is a strictly local minimum of
$\Phi$ on $\mathcal{N}$. Similarly, we can prove that $\mathbf{u}_2$ is a
strictly local minimum of $\Phi$ on $\mathcal{N}$.
\end{proof}

{\textbf{\emph{Proof of Theorem \ref{thm1}}}}

Lemma \ref{2.1} and Lemma \ref{strict local min} imply that
$\Phi_{\mathcal{N}}$ satisfies the hypotheses of the Mountain Pass
Lemma (\cite{Ambrosetti-Rabinowitz73}). Hence the functional $\Phi$
has a mountain-pass point $\mathbf{u}^*$ on $\mathcal{N}$ such that
$\Phi(\mathbf{u}^*)> max\{\Phi(\mathbf{u}_1),\Phi(\mathbf{u}_2)\}$.

We claim that the Morse index of $\mathbf{u}^{*}$ is at least two.
Indeed, it follows from  $(\ref{negative definite of F})$ that
$$\langle\Phi^{''}(\mathbf{u})\mathbf{u},\mathbf{u}\rangle
<0,\; \forall \mathbf{u}\in\mathcal{N},$$
which implies that the Morse index of any critical point of $\Phi$
equals its Morse index as constrained critical point of $\Phi_{\mathcal{N}}$, increased by 1.
Hence, the Morse index of $\mathbf{u}$  is at least two.

 In order to prove Theorem \ref{thm1}, it is enough to ensure that
 $\mathbf{u}^*>0$. We discuss it as follows.

 If $\beta>0,$ we introduce the functional
\begin{eqnarray*}
  \Phi^+(\mathbf{u}) &=& \frac{1}{2}(a_1\|u_1\|^2+a_2\|u_2\|^2)+
    \frac{1}{4}(b_1\|u_1\|^4+b_2\|u_2\|^4) \\
   &&-\frac{1}{p+1}\int_{\Omega}(\lambda_1|u_1^+|^{p+1}+\lambda_2|u_2^+|^{p+1}
    +2\beta|u_1^+|^{\frac{p+1}{2}}|u_2^+|^{\frac{p+1}{2}})\,\mathrm{d}x,
\end{eqnarray*}
where $u^+=\max\{u,0\}$ and the corresponding Nehari manifold
\begin{equation*}
    \mathcal{N}^+\stackrel{\mathrm{def}}{=}\{\;\mathbf{u}\in\mathcal{H}\setminus\{\mathbf{0}\}:
    \quad\langle\nabla\Phi^+(\mathbf{u}),\mathbf{u}\rangle=0\}.
\end{equation*}

 Since $3<p<2^*-1$, $\Phi^+\in\mathcal{C}^2(\mathcal{H},\mathbb{R})$.
 Hence we can repeat the arguments as above and get a mountain-pass
 critical point $\mathbf{u}^*$ of $\Phi^+$ on $\mathcal{N}^+$, which gives rise to
 a solution of
 \begin{equation}\label{positive system}
 \begin{cases}
 -(a_1+b_1\int_{\Omega} |\nabla u_1|^2\,\mathrm{d}x)\Delta u_1\ =
  \lambda_1 (u_1^+)^p+\beta (u_1^+)^{\frac{p-1}{2}} (u_2^+)^{\frac{p+1}{2}} \qquad in \ \Omega\\
-(a_2+b_2\int_{\Omega}|\nabla u_2|^2\,\mathrm{d}x)\Delta u_2\ =
\lambda_2 (u_2^+)^p+\beta (u_1^+)^{\frac{p+1}{2}} (u_2^+)^{\frac{p-1}{2}} \qquad in \ \Omega\\
u_1=u_2=0 \quad\, on\ \partial\Omega.
\end{cases}
\end{equation}
 It is clear to see that $\mathbf{u}^*\ge 0$. The fact that
 $\mathbf{u}^*\in\mathcal{N}^+$ implies $\mathbf{u}^*\neq \mathbf{0}$.
 By the uniqueness of problem (\ref{2}) with $f(x,u)=\lambda(u^+)^p$ (see Remark \ref{rem 1.1}), we get that $u_1^*\not\equiv
 0$ and $u_2^*\not\equiv 0$. Applying the maximum principle to each
 equation in (\ref{positive system}), we get that $u_1^*> 0$ and $u_2^*>
 0$. This completes the proof of Theorem \ref{thm1} when $\beta>0$.

 If $-\sqrt{\lambda_1\lambda_2}<\beta\le 0$, we consider the following
 functional
\begin{eqnarray*}
   \tilde\Phi(\mathbf{u})&=& \frac{1}{2}(a_1\|u_1\|^2+a_2\|u_2\|^2)+
    \frac{1}{4}(b_1\|u_1\|^4+b_2\|u_2\|^4) \\
   && -\frac{1}{p+1}\int_{\Omega}(\lambda_1|u_1^+|^{p+1}+\lambda_2|u_2^+|^{p+1}
    +2\beta|u_1|^{\frac{p+1}{2}}|u_2|^{\frac{p+1}{2}})\,\mathrm{d}x
\end{eqnarray*}
and the Nehari manifold
\begin{equation*}
    \tilde{\mathcal{N}}\stackrel{\mathrm{def}}{=}\{\;\mathbf{u}\in\mathcal{H}\setminus\{\mathbf{0}\}:
    \quad\langle\nabla\tilde{\Phi}(\mathbf{u}),\mathbf{u}\rangle=0\}.
\end{equation*}
Similar arguments imply that $\tilde\Phi$ has a mountain-pass
critical point $\mathbf{u}^*$ of $\tilde\Phi$ on
$\tilde{\mathcal{N}}$, which give rise to a solution of
 \begin{equation*}
 \begin{cases}
 -(a_1+b_1\int_{\Omega} |\nabla u_1|^2\,\mathrm{d}x)\Delta u_1\ -
 \beta |u_1|^{\tfrac{p-3}{2}} |u_2|^{\tfrac{p+1}{2}}u_1=\lambda_1 (u_1^+)^p \qquad in \ \Omega,\\
-(a_2+b_2\int_{\Omega} |\nabla u_2|^2\,\mathrm{d}x)\Delta u_2\ -
\beta |u_1|^{\tfrac{p+1}{2}} |u_2|^{\tfrac{p-3}{2}}u_2=
\lambda_2 (u_2^+)^p \qquad in \ \Omega,\\
u_1=u_2=0 \quad\, on\ \partial\Omega.
\end{cases}
\end{equation*}
Multiplying these equations with $u_1^{-}$ resp. $u_2^-$ and
integrating, we get
\begin{eqnarray}\label{two equ for nonnegative}
(a_1+b_1\|u_1\|^2)\int_{\Omega} |\nabla u_1^-|^2\,\mathrm{d}x-
\beta \int_{\Omega}|u_1^-|^{\tfrac{p+1}{2}}
|u_2|^{\tfrac{p+1}{2}}\,\mathrm{d}x&=&0,\\
(a_2+b_2\|u_2\|^2)\int_{\Omega} |\nabla
u_2^-|^2\,\mathrm{d}x-\beta\int_{\Omega} |u_1|^{\tfrac{p+1}{2}}
|u_2^-|^{\tfrac{p+1}{2}}\,\mathrm{d}x&=&0.\label{two equ for nonnegative 2}
\end{eqnarray}
 Since
$-\sqrt{\lambda_1\lambda_2}<\beta\le0$, we conclude that $u_1^*\ge
0$ and $u_2^*\ge
 0$. By the uniqueness of problem (\ref{2}) with $f(x,u)=\lambda(u^+)^p$ (see Remark \ref{rem 1.1}), $u_1^*\not\equiv 0$ and
 $u_2^*\not\equiv 0$. Hence $u_1^*> 0$ and $u_2^*>
 0$ in $\Omega$ by the strong maximum principle.

 This completes the proof of Theorem \ref{thm1}.
 \hspace{\stretch{1}}$\Box$

\section{The proof of Theorem \ref{thm2}}
Assume that the hypotheses in Theorem $\ref{thm2}$ hold  throughout the section, which lead us to consider the
following nonlocal problem
 \begin{equation}\label{fully sym equation}
 \begin{cases}
 -(a+b\int_{\Omega} |\nabla u_1|^2\,\mathrm{d}x)\Delta u_1\ -
 \beta |u_1|^{\tfrac{p-3}{2}} |u_2|^{\tfrac{p+1}{2}}u_1=\lambda(u_1^+)^p \qquad in \ \Omega,\\
-(a+b\int_{\Omega} |\nabla u_2|^2\,\mathrm{d}x)\Delta u_2\ - \beta
|u_1|^{\tfrac{p+1}{2}} |u_2|^{\tfrac{p-3}{2}}u_2=\lambda
(u_2^+)^p \qquad in \ \Omega,\\
u_1=u_2=0 \quad\, on\ \partial\Omega,
\end{cases}
\end{equation}
where $a,b,\lambda$ are positive constants, $\beta\le-\lambda$
and $p\in(3,2^{*}-1).$

The energy functional associated with (\ref{fully sym equation}) ( also denoted by $\Phi$) is
\begin{eqnarray*}
    \Phi(\mathbf{u})&=&\frac{a}{2}(\|u_1\|^2+\|u_2\|^2)+
    \frac{b}{4}(\|u_1\|^4+\|u_2\|^4)\\
    &&-\frac{1}{p+1}\int_{\Omega}\left(\lambda|u_1^+|^{p+1}+\lambda|u_2^+|^{p+1}
    +2\beta|u_1|^{\frac{p+1}{2}}|u_2|^{\frac{p+1}{2}}\right)\,\mathrm{d}x,
\end{eqnarray*}
and clearly, $\Phi\in\mathcal{C}^2(\mathcal{H},\mathbb{R}).$
Then, we put
$$
\mathcal{M}\stackrel{\mathrm{def}}{=}\left\{(u_1,u_2)\in\mathcal{H}\left|\begin{array}{cc}
u_1\not\equiv 0,&
a\|u_1\|^2+b\|u_1\|^4-\beta\int_{\Omega}|u_1|^{\frac{p+1}{2}}|u_2|^{\frac{p+1}{2}}\,\mathrm{d}x
=\lambda|u_1^+|_{p+1}^{p+1}\\
u_2\not\equiv 0,&a\|u_2\|^2+b\|u_2\|^4-\beta\int_{\Omega}|u_1|^{\frac{p+1}{2}}|u_2|^{\frac{p+1}{2}}\,\mathrm{d}x
=\lambda|u_2^+|_{p+1}^{p+1}\end{array}\right.
\right\},
$$ which possesses the following properties:
\begin{lem}\label{M}
$(i)$ There exists $\rho>0$ such that $\|\mathbf{u}\|\ge \rho$ for
any
 $\mathbf{u}\in\mathcal{M}$.

$(ii)$ $\mathcal{M}$ is a $\mathcal{C}^2$ complete manifold of
codimension two in $\mathcal{H}$.

$(iii)$ If $\mathbf{u}$ is a critical point of the restriction
$\Phi_{\mathcal{M}}$ of $\Phi$ to $\mathcal{M}$, then $\mathbf{u}$
is a nontrivial critical point of $\Phi$.

$(iv)$ $\Phi(\mathbf{u})=\frac{a(p-1)}{2(p+1)}(\|u_1\|^2+\|u_2\|^2)$
$+
$$\frac{b(p-3)}{4(p+1)}(\|u_1\|^4+\|u_2\|^4)$
, for $\mathbf{u}\in \mathcal{M}$.

 $(v)$ \label{PS}$\Phi_{\mathcal{M}}$ satisfies Palais-Smale condition.
\end{lem}
\begin{proof}
(i)By Sobolev embedding theorem, for all $\mathbf{u}=(u_1,u_2)\in\mathcal{M}$ we have that
$$a\|u_i\|^2\le\lambda|u_i^{+}|^{p+1}_{p+1}\le\lambda C\|u_i^{+}\|^{p+1}$$ for $i=1,2$, and hence,
$$\|\mathbf{u}\|\ge\max\{\|u_1\|,\|u_2\|\}\ge\rho:=(\frac{a}{\lambda C})^{\frac{1}{p+1}},$$
for all $\mathbf{u}\in\mathcal{M}$.

(ii)We define $F:\mathcal{H}\to \mathbb{R}^2$ by
\begin{equation*}
    F(\mathbf{u})=\left(\begin{array}{c}
    F_1(\mathbf{u})\\F_2(\mathbf{u})
    \end{array}\right)
    =\left(
    \begin{array}{c}
    a\|u_1\|^2+b\|u_1\|^4-\beta\int_{\Omega}|u_1|^{\frac{p+1}{2}}|u_2|^{\frac{p+1}{2}}\,\mathrm{d}x
-\lambda|u_1^+|_{p+1}^{p+1}\\
a\|u_2\|^2+b\|u_2\|^4-\beta\int_{\Omega}|u_1|^{\frac{p+1}{2}}|u_2|^{\frac{p+1}{2}}\,\mathrm{d}x
-\lambda|u_2^+|_{p+1}^{p+1}
    \end{array}
    \right).
\end{equation*}
Then $F\in \mathcal{C}^2$, and
\begin{equation*}
    \mathcal{M}=\left\{
    \mathbf{u}\in\mathcal{H}\left|
    u_1,u_2\not\equiv 0, F(\mathbf{u})=0
    \right.
    \right\}.
\end{equation*}

We shall prove that for all $\mathbf{u}=(u_1,u_2)\in\mathcal{M}$,\,
$F^{'}(\mathbf{u})(u_1,0)$ and $F^{'}(\mathbf{u})(0,u_2)$ are linearly independent in $\mathbb{R}^2$.
It is enough for us to prove that the matrix
\begin{equation*}
    \left(
    \begin{array}{cc}
    \partial_{u_1}F_1(\mathbf{u})u_1&\partial_{u_2}F_1(\mathbf{u})u_2\\
        \partial_{u_1}F_2(\mathbf{u})u_1&\partial_{u_2}F_2(\mathbf{u})u_2
    \end{array}
    \right)
\end{equation*}
is negative definite.
In fact, for $\mathbf{u}\in\mathcal{M}$ we have
\begin{eqnarray*}
    \partial_{u_1}F_1(\mathbf{u})u_1&=&2a\|u_1\|^2+4b\|u_1\|^4-\frac{p+1}{2}\beta\int_{\Omega}|u_1|^{\frac{p+1}{2}}|u_2|^{\frac{p+1}{2}}\,\mathrm{d}x
    -(p+1)\lambda|u_1^{+}|^{p+1}_{p+1}\\
    &=&(1-p)a\|u_1\|^2+(3-p)b\|u_1\|^4+\frac{p+1}{2}\beta\int_{\Omega}|u_1|^{\frac{p+1}{2}}|u_2|^{\frac{p+1}{2}}\,\mathrm{d}x<0,\\
\partial_{u_2}F_2(\mathbf{u})u_2&=&(1-p)a\|u_2\|^2+(3-p)b\|u_2\|^4+\frac{p+1}{2}\beta\int_{\Omega}|u_1|^{\frac{p+1}{2}}|u_2|^{\frac{p+1}{2}}\,\mathrm{d}x<0.
\end{eqnarray*}
Notice that
\begin{eqnarray*}
&&-\partial_{u_1}F_1(\mathbf{u})u_1>-\frac{p+1}{2}\beta\int_{\Omega}|u_1|^{\frac{p+1}{2}}|u_2|^{\frac{p+1}{2}}\,\mathrm{d}x=\partial_{u_2}F_1(\mathbf{u})u_2,\\ &&-\partial_{u_2}F_2(\mathbf{u})u_2>-\frac{p+1}{2}\beta\int_{\Omega}|u_1|^{\frac{p+1}{2}}|u_2|^{\frac{p+1}{2}}\,\mathrm{d}x=\partial_{u_1}F_2(\mathbf{u})u_1,\\
&&and \quad\partial_{u_1}F_1(\mathbf{u})u_1,\;\partial_{u_2}F_2(\mathbf{u})u_2<0.
\end{eqnarray*}
It follows that the matrix
\begin{equation*}
    \left(
    \begin{array}{cc}
    \partial_{u_1}F_1(\mathbf{u})u_1&\partial_{u_2}F_1(\mathbf{u})u_2\\
        \partial_{u_1}F_2(\mathbf{u})u_1&\partial_{u_2}F_2(\mathbf{u})u_2
    \end{array}
    \right)
\end{equation*}
is negative definite. Consequently, $F^{'}(\mathbf{u})(u_1,0)$ and
$F^{'}(\mathbf{u})(0,u_2)$ are linearly independent in
$\mathbb{R}^2$, and hence
$F^{'}(\mathbf{u}):\mathcal{H}\to\mathbb{R}^2$ is surjective at
every point $\mathbf{u}\in\mathcal{M}$. This jointly with Implicit
Function Theorem implies that $\mathcal{M}$ is a $\mathcal{C}^2$
complete manifold of codimension two in $\mathcal{H}$.

(iii)
If $\mathbf{u}$ is a critical point of $\Phi_{\mathcal{M}}$,
then there exist $\omega_1=\omega_1(\mathbf{u})$ and $\omega_2=\omega_2(\mathbf{u})$ in $\mathbb{R}$ such that
\begin{equation}\label{the equ in lemma3.1}
    \Phi^{'}(\mathbf{u})=\omega_1F_1^{'}(\mathbf{u})+\omega_2 F^{'}_2(\mathbf{u}).
\end{equation}
Multiplying this with $(u_1,0)$ and $(0,u_2)$, respectively, we get
\begin{equation*}
    \left(
    \begin{array}{cc}
    \partial_{u_1}F_1(\mathbf{u})u_1&\partial_{u_2}F_1(\mathbf{u})u_2\\
        \partial_{u_1}F_2(\mathbf{u})u_1&\partial_{u_2}F_2(\mathbf{u})u_2
    \end{array}
    \right)
    \left(
    \begin{array}{c}
    \omega_1\\
      \omega_2
    \end{array}
    \right)
    =
    \left(
    \begin{array}{c}
    0\\
    0
    \end{array}
    \right).
\end{equation*}
It follows from the matrix is negative definite that $\omega_1=\omega_2=0$,
and therefore, $\Phi^{'}(\mathbf{u})=0$ by (\ref{the equ in lemma3.1}).

(iv)
This results from a simple calculation by the definition of
$\mathcal{M}.$

(v)
 Assume that $\{\mathbf{u}_n\} \subset\mathcal{M}$ is a $(PS)_c$ sequence
of $\Phi_{\mathcal{M}}$, that is,
\begin{equation*}
\Phi(\mathbf{u}_n)\to c  \quad and
 \quad \nabla_{\mathcal{M}}\Phi(\mathbf{u}_n)\to \mathbf{0}.
\end{equation*}
Then there exist two sequences $\{\omega_{n,1}\}$ and $\{\omega_{n,2}\}$ in $\mathbb{R}$ such that
\begin{equation}\label{the derivative of Psi}
    \nabla_{\mathcal{M}}\Phi(\mathbf{u}_n)=\Phi^{'}(\mathbf{u}_n)-\omega_{n,1}
F_1^{'}(\mathbf{u}_n)-\omega_{n,2}F_2^{'}(\mathbf{u}).
\end{equation}
By  $(iv)$ and $\Phi(\mathbf{u}_n)\to c$ we have that
$\|\mathbf{u}_n\|\le C<+\infty$.
Going if necessary to a subsequence, we can assume that $\mathbf{u}_n\rightharpoonup
\mathbf{u}$ in $\mathcal{H}$ and $\mathbf{u}_n \to \mathbf{u}$ in
$\mathcal{L}^{p+1}(\Omega)$.

Multiplying $(\ref{the derivative of Psi})$ with $(u_{n,1},0)$
 and $(0,u_{n,2})$, we have
 \begin{equation}\label{the equ about Matrix Mn}
 M_n
 \left(
    \begin{array}{c}
    \omega_{n,1}\\
      \omega_{n,2}
    \end{array}
    \right)
        =o(1).
\end{equation}
where the matrix
\begin{equation*}
    M_n=
    \left(
    \begin{array}{cc}
    M_{n}^{(1)}&M_{n}^{(2)}\\
        M_{n}^{(2)}&M_{n}^{(3)}
    \end{array}
    \right):=
    \left(
    \begin{array}{cc}
    \partial_{u_1}F_1(\mathbf{u}_n)u_{n,1}&\partial_{u_2}F_1(\mathbf{u}_n)u_{n,2}\\
        \partial_{u_1}F_2(\mathbf{u}_n)u_{n,1}&\partial_{u_2}F_2(\mathbf{u}_n)u_{n,2}
    \end{array}
    \right).
\end{equation*}
We assume that $\|u_{n,1}\| \to T_1$ and $\|u_{n,2}\|\to T_2$ as $n\to+\infty$,
and then $T_1,\,T_2>0$ by (i).
Moreover,
\begin{eqnarray*}
  M^{(1)}_{n}&\to&M^{(1)}:=(1-p)a T_1^2+(3-p)b T_2^4+\frac{p+1}{2}\beta\int_{\Omega}|u_1|^{\frac{p+1}{2}}|u_2|^{\frac{p+1}{2}}\,\mathrm{d}x,\\
  M^{(3)}_{n}&\to&M^{(3)}:=(1-p)a T_2^2+(3-p)b T_2^4+\frac{p+1}{2}\beta\int_{\Omega}|u_1|^{\frac{p+1}{2}}|u_2|^{\frac{p+1}{2}}\,\mathrm{d}x,\\
M^{(2)}_n&\to&M^{(2)}:=-\frac{p+1}{2}\beta\int_{\Omega}|u_1|^{\frac{p+1}{2}}|u_2|^{\frac{p+1}{2}}\,\mathrm{d}x.
\end{eqnarray*}
If we denote
\begin{equation*}
    M=
    \left(
    \begin{array}{cc}
    M^{(1)}&M^{(2)}\\
        M^{(2)}&M^{(3)}
    \end{array}
    \right),
\end{equation*}
then
$M$ is negative definite and
$$
(M+o(1))
\left(
    \begin{array}{c}
\omega_{n,1}\\
\omega_{n,2}
    \end{array}
    \right)
    =o(1)
$$
by (\ref{the equ about Matrix Mn}).
We therefore conclude that $\omega_{n,1}\to 0$  and $\omega_{n,2}\to 0$.
By similar arguments about the estimate of $\|F^{'}(\mathbf{u})\|$ in Lemma \ref{2.1} (v), we can prove that
$$\|F_1^{'}(\mathbf{u}_n)\|\le C \qquad and\qquad \|F_2^{'}(\mathbf{u}_n)\|\le C.
$$

It follows from $(\ref{the derivative of Psi})$ that
$\Phi^{'}(\mathbf{u}_n)\to 0$, that is, $\mathbf{u}_n$ is a $(PS)_c$ sequence of $\Phi$.
Notice that
\begin{eqnarray*}\label{PScondition2}
\nonumber&&(a+b\|u_{n,1}\|^2)\int_{\Omega}\nabla
u_{n,1}\cdot\nabla(u_{n,1}-u_1)\mathrm{d}\,x
+(a+b\|u_{n,2}\|^2)\int_{\Omega}\nabla u_{n,2}\cdot\nabla(u_{n,2}-u_2)\mathrm{d}\,x\\
   &=&\langle \Phi^{'}(\mathbf{u}_n),\mathbf{u}_n-\mathbf{u}\rangle
+\int_{\Omega}[\lambda_1 |u_{n,1}^{+}|^{p-1}u_{n,1}(u_{n,1}-u_1)+\lambda_2
|u_{n,2}^{+}|^{p-1}u_{n,2}(u_{n,2}-u_2)]\;\mathrm{d}\,x\nonumber \\
&&+\beta\int_{\Omega}[|u_{n,1}|^{\frac{p-3}{2}}|u_{n,2}|^{\frac{p+1}{2}}u_{n,1}(u_{n,1}-u_1)+
|u_{n,1}|^{\frac{p+1}{2}}|u_{n,2}|^{\frac{p-3}{2}}u_{n,2}(u_{n,2}-u_2)]\;\mathrm{d}x\nonumber\\
   &\to&0,
\end{eqnarray*}
and \begin{equation*}
 \lim_{n\to\infty}\int_{\Omega}\nabla u_{n,1}\cdot\nabla(u_{n,1}-u_1)\;\mathrm{d}\,x \ge 0,\;\lim_{n\to\infty}\int_{\Omega}\nabla u_{n,2}\cdot\nabla(u_{n,2}-u_2)\;\mathrm{d}\,x\ge 0,
\end{equation*} by Fatou's Lemma.
It follows that
\begin{eqnarray*}
     (a_1+b_1\|u_{n,1}\|^2)\int_{\Omega}\nabla u_{n,1}\cdot\nabla(u_{n,1}-u_1)\;\mathrm{d}\,x
 &\to& 0,\\
   (a_2+b_2\|u_{n,2}\|^2)\int_{\Omega}\nabla u_{n,2}\cdot\nabla(u_{n,2}-u_2)\;\mathrm{d}\,x&\to& 0.
\end{eqnarray*}
This jointly with the fact that $\|u_{n,i}\|$ is bounded, yields
that $\|u_{n,i}\|\to \|u_{i}\|$, and hence, $u_{n,i}\to u_i$ in $H$ for
$i=1,2$, that is, $\mathbf{u}_n\to \mathbf{u}$ in $\mathcal{H}$.
Moreover, we have that $\mathbf{u}\in\mathcal{M}$  by
$\mathbf{u}_n\in\mathcal{M}$.
\end{proof}
One can write the group $\mathbb{Z}_2$  multiplicatively as $\{1,
\sigma\}$, where $\sigma:\mathcal{H}\to\mathcal{H}$ is defined by
$\sigma(u_1,u_2)=(u_2,u_1)$. Then $\mathcal{M}$ is a free
$\mathbb{Z}_2-$space $( \sigma(\mathbf{u})\not=\mathbf{u}, \forall
\mathbf{u}\in\mathcal{M}))$ since $\beta\le -\lambda$.

Let $\mathcal{F}$ denote the class of $\sigma-$invariant
closed subsets of $\mathcal{M}$. The cohomological index
$i:\mathcal{F}\to\mathrm{N}\cup\{0,\infty\}$ due to Fadell and
Rabinowitz\cite{Fadell-Rabinowitz1978} is  well defined and
satisfies the following properties.
\begin{lem}(\cite{Fadell-Rabinowitz1978})\label{properties of i}\qquad$\forall$ $A,B\in\mathcal{F}$,

$(i_1)$(Definiteness) $i(A)=0$ if and only if $A=\emptyset$;

$(i_2)$(Monotonicity) If $f:A\to B$ is an $\sigma-$equivariant map
(in particular, if $A\subset B$), then $i(A)\le i(B)$;

$(i_3)$(Continuity) If $A$ is compact then $i(A)<\infty$ and there
exists a relatively open $\sigma-$invariant neighborhood $N$ of $A$
in $\mathcal{M}$ such that $i(A)=i(N)$;

$(i_4)$(Subadditivity) If $X\in\mathcal{F}$ and $X=A\cup B$, then
$i(X)\le i(A)+i(B)$;

$(i_5)$(Neighborhood of zero) If $U$ is a bounded closed
$\sigma-$invariant neighborhood of the origin in a normed  linear
space $W$, then $i(\partial U)$ equals to the dimension of $W$.
\end{lem}

Denote $K_c=\{\mathbf{u}\in
\mathcal{M}:\Phi(\mathbf{u})=c,\Phi^{'}(\mathbf{u})=0\}$. For
$k\in\mathrm{N}$, let $\mathcal{F}_k=\{M\in\mathcal{F}: \quad
i(M)\ge k\}$  and
\begin{equation*}
    c_k=\inf_{M\in\mathcal{F}_k}\sup_{\mathbf{u}\in
M}\Phi(\mathbf{u}).
\end{equation*}
Since $\mathcal{F}_{k+1}\subset\mathcal{F}_k$, $c_k\le c_{k+1}$.

We claim that $c_k$ is finite. In fact, for fixed $k$, there exist
$2k$ points $\{x_1,x_2,\cdots,x_{2k} \}\subset\mathring{\Omega},$ and
$\rho_k>0$ such that $\bigcup^{2k}_{i=1}B_{\rho_k}(x_i)\subset\Omega$ and
$B_{\rho_{k}}(x_i)\bigcap
B_{\rho_{k}}(x_j)=\emptyset,$ for any $1\le i\not=j\le 2k.$ Define
cut-off functions by $\phi_i(x):=\phi(\frac{x-x_i}{\rho_k}),$
$i=1,\cdots,2k,$ where $\phi\in\mathcal{C}^{\infty}_0(\Omega,[0,1])$
satisfying $\phi\equiv 1$ on $B_{\frac{1}{2}}(0)$, $\phi\equiv 0$ on
$\Omega\setminus B_1(0)$ and $|\nabla\phi|_{\infty}\le 2$. Then
$\mathbf{v}_i:=(\phi_i,\phi_{k+i})\in\mathcal{H},$ $i=1,\cdots,k,$
and hence, $W:=$span$\{\mathbf{v}_1,\cdots,\mathbf{v}_k\}$ is a $k$
dimensional subspace of $\mathcal{H}$. We can deduce from arguments below
$(\ref{the
fuc of f})$ that  there exists a unique
$t(\mathbf{w})\in\mathbb{R}^{+}$ such that
$t(\mathbf{w})\mathbf{w}\in\mathcal{M}.$ It follows from
 $(i_5)$ of Lemma \ref{properties of i} that the intersection of $\mathcal{M}$ with $W$ is a compact set
 in $\mathcal{F}_k.$
 Hence, $c_k\le max\{\Phi(\mathbf{v})\,|\, \mathbf{v}\in\mathcal{M}\bigcap W\}<+\infty$.

 Notice that $\Phi$ is $\sigma-$invariant and $\Phi_{M}$
satisfies $(PS)$ condition. We have
\begin{prop}\label{propostion }
(1) If $-\infty<c_k=\cdots=c_{k+m-1}=c<\infty$, then $i(K_c)\ge m$.
In particular, if $-\infty<c_k<\infty$ then $K_{c_k}\not=\emptyset$.

(2) If $-\infty<c_k<\infty$ for all sufficiently large $k$, then
$c_k\to\infty$ as $k\to +\infty$.
\end{prop}
This proposition is a slight variant of Proposition 3.14.7 in
\cite{Perera}. We sketch its proof for completeness. Firstly, two
lemmas are in order.

Since $\Phi$ and $\mathcal{M}$ are $\sigma-$invariant, by Lemma $3.1
$ of \cite{Willem} we have
\begin{lem}(Equivariant Deformation Lemma)\label{first deformation lemma}

If $c\in\mathbb{R}$ and $\delta>0$, then there exist $\varepsilon>0$
and a map $\eta\in\mathcal{C}([0,1]\times\mathcal{M},\mathcal{M})$
satisfying

$(i)$ $\eta(0,\mathbf{u})=\mathbf{u}$,  \quad$\forall\;
\mathbf{u}\in\mathcal{M}$;

$(ii)$ $\eta(1,\Phi^{c+\varepsilon}\setminus
N_{\delta}(K_c))\subset\Phi^{c-\varepsilon}$;

$(iii)$ $\Phi(\eta(\cdot,\mathbf{u}))$ is non-increasing for all
$\mathbf{u}\in\mathcal{M}$;

$(iv)$ $\sigma(\eta(s,\mathbf{u}))=\eta(s,\sigma(\mathbf{u})),$\quad
$\forall s\in[0,1]$ and $\mathbf{u}\in\mathcal{M}$.
\end{lem}
\begin{lem}\label{3.3}
Let $c\in \mathbb{R}$. If there is  $\varepsilon>0$ such that
$c-\varepsilon<c_k\le\cdots\le c_{k+m-1}<c+\varepsilon$, then
$i(K_c)\ge m$.
\end{lem}
\begin{proof}
Since $\Phi_{\mathcal{M}}$ satisfies $(PS)$ condition, $K_c$ is
compact. Then $i(N_{\delta}(K_c))=i(K_c)$  for some $\delta>0$ by
Lemma \ref{properties of i} $(i_3)$. Lemma \ref{first deformation
lemma} implies that there exist  $\varepsilon>0$ and a
$\sigma-$equivariant map
$\eta\in\mathcal{C}([0,1]\times\mathcal{M},\mathcal{M})$ such that
$\eta(\Phi^{c+\varepsilon}\setminus
N_{\delta}(K_c))\subset\Phi^{c-\varepsilon}$. It follows that
\begin{equation}\label{equ of lem 3.3}
    i(\Phi^{c+\varepsilon})\le i(\overline{\Phi^{c+\varepsilon}\setminus N_{\delta}(K_c)})
+i(N_{\delta}(K_c))\le i(\Phi^{c-\varepsilon})+i(K_c)
\end{equation}
by Lemma \ref{properties of i} $(i_2)$ and $(i_4)$. The fact that
$c-\varepsilon<c_k$ jointly the definition of $c_k$ yields that
$i(\Phi^{c-\varepsilon})\le k-1$. Similarly, we have that
$i(\Phi^{c+\varepsilon})\ge k+m-1$. So $i(K_c)\ge m$ follows from
(\ref{equ of lem 3.3}).
\end{proof}
{\textbf{\emph{Proof of Proposition \ref{propostion }}}}

$(1)$ Lemma \ref{3.3} immediately yields that $i(K_c)\ge m$. If one
take $m=1$ then $i(K_{c_k})\ge 1$. So $K_{c_k}\not=\emptyset$ by
Lemma \ref{properties of i} $(i_1)$.

$(2)$ Suppose  by contradiction that $c_k\to {c}<\infty$  as
$k\to\infty$. Choosing $\varepsilon>0$ as in Lemma \ref{3.3} and $k$
so large that $c_k>c-\varepsilon$, one has that $i(K_c)=\infty$
since $c_{k+m-1}\le c$ for all $m$, contradicting with the
compactness of $K_c$ and Lemma \ref{properties of i} $(i_4)$.
 \hspace{\stretch{1}}$\Box$

Now we complete the proof of Theorem \ref{thm2}.

{\textbf{\emph{Proof of Theorem \ref{thm2}}}}

  Proposition \ref{propostion } implies that one can take
$\mathbf{u}_k\in K_{c_k}$ for every $k$ satisfying
$\Phi(\mathbf{u}_k)\to\infty$  as $k\to\infty$. By Lemma
\ref{M}$(iv)$, we get that $\|\mathbf{u}_k\|\to\infty$ as
$k\to\infty$.

Since $\mathbf{u}_k\in\mathcal{M}$,
\begin{equation*}
    a\|\mathbf{u}_k\|^2\le\lambda|\mathbf{u}_k^+|_{p+1}^{p+1}
    \le\lambda|\Omega|^{p+1}\|\mathbf{u}_k\|_{\mathcal{L}^{\infty}(\Omega)}^{p+1}.
\end{equation*}
Therefore, $\|\mathbf{u}_k\|_{\mathcal{L}^{\infty}(\Omega)}\to
\infty$, as $k\to\infty$.
Next, we shall prove $\mathbf{u}_k>0.$
It follows from $(\ref{two equ for nonnegative})$ and $(\ref{two equ for nonnegative 2})$
that $\mathbf{u}_k\ge 0$.
Since $\mathbf{u}_k=(u_{k,1},u_{k,2})\in\mathcal{M}$, $u_{k,1}\not\equiv 0$
and $u_{k,2}\not\equiv 0$. Thus, the strong maximum principle implies that
$u_{k,1}>0$ and $u_{k,2}>0$ in $\Omega.$
This completes the proof of Theorem
\ref{thm2}. \hspace{\stretch{1}}$\Box$
\begin{rek}
 For problem (\ref{BECS}), Dancer-Wei-Weth established similar results in \cite{DWW2010}. However, their proof relies on a variant of
 Liusternik-Schnirelman theory on a sub-manifold.
\end{rek}

\section{Proofs of Theorem \ref{thm3} and Corollary \ref{cor 1.4}}

In this section, we deal with the following problem
\begin{equation}\label{critical system}
\begin{cases}
 -(a_1+b_1\int_{\Omega} |\nabla u_1|^2\,\mathrm{d}x)\Delta u_1\ =
  \lambda_1 u_1^5+ \beta |u_1|^{\tfrac{q-3}{2}} |u_2|^{\tfrac{q+1}{2}}u_1 \qquad in \ \Omega,\\
-(a_2+b_2\int_{\Omega} |\nabla u_2|^2\,\mathrm{d}x)\Delta u_2\ =
 \lambda_2 u_2^5+ \beta |u_1|^{\tfrac{q+1}{2}} |u_2|^{\tfrac{q-3}{2}}u_2 \qquad in \ \Omega,\\
u_1,u_2>0  \quad in \; \Omega,\qquad u_1=u_2=0 \quad on\ \partial\Omega,
\end{cases}
\end{equation}
where $\Omega\subset\mathbb{R}^3$ is a smooth bounded domain,
$\beta\in\mathbb{R}^{+}$ and $q\in(3,5)$.

 The solutions of problem $(\ref{critical system})$ are critical points of the following functional on $\mathcal{H}$:
\begin{eqnarray*}
  \Psi(\mathbf{u}) &=& \frac{1}{2}(a_1\|u_1\|^2+a_2\|u_2\|^2)+
    \frac{1}{4}(b_1\|u_1\|^4+b_2\|u_2\|^4) \\
  &&-    \frac{1}{6}\int_{\Omega}\left(\lambda_1|u_1^+|^6+\lambda_2|u_2^+|^6\right)\,\mathrm{d}x-\frac{2\beta}{q+1}\int_{\Omega}|u_1^+|^{\frac{q+1}{2}}|u_2^+|^{\frac{q+1}{2}}\,\mathrm{d}x.
\end{eqnarray*}
We firstly prove that this functional has the Mountain Pass
Geometry.
\begin{lem}\label{Mountainpass Geometry}
For $\beta>0$, there exist  $\delta$,$\alpha>0$ and $\mathbf{e}\in\mathcal{H}$
satisfying that

(i) for any $\mathbf{u}\in\mathcal{H}$ with
$\|\mathbf{u}\|=\delta$, $\Psi(\mathbf{u})\ge\alpha>0$;

(ii) $\|\mathbf{e}\|>\delta$ and $\Psi(\mathbf{e})<0$.
\end{lem}
\begin{proof}
 By Sobolev  embedding theorem, there exists $C>0$ such that
\begin{equation*}
|u^+|^6_6\le C\|u\|^6,\quad |u^+|^{q+1}_{q+1}\le C\|u\|^{q+1}, \quad
\forall u\in H.
\end{equation*}
Hence, we obtain
\begin{eqnarray*}
  \Psi(\mathbf{u})&\ge&\frac{1}{2}(a_1\|u_1\|^2+a_2\|u_2\|^2)+
    \frac{1}{4}(b_1\|u_1\|^4+b_2\|u_2\|^4)\\
    &&-\frac{(\lambda_1+\lambda_2)C}{6}(\|u_1\|^6+\|u_2\|^6)
    -\frac{\beta C}{q+1}(\|u_1\|^{q+1}+\|u_2\|^{q+1}).
\end{eqnarray*}
It follows from $q\in(3,5)$ that  there exists $\delta>0$ such that
\begin{equation*}
    \alpha:=\inf_{\|\mathbf{u}\|=\delta}\Psi(\mathbf{u})>0.
\end{equation*}

 Let $\phi\in$ $H$ with $\phi>0$ in $\Omega$. We have, for $t\ge
0$
\begin{equation*}
    \Psi(t\phi,t\phi)=\frac{a_1+a_2}{2}\|\phi\|^2
    t^2+\frac{b_1+b_2}{4}\|\phi\|^4 t^4
    -\frac{\lambda_1+\lambda_2}{6}|\phi|^6_6 t^6-
    \frac{2\beta}{q+1}|\phi|^{q+1}_{q+1} t^{q+1}.
\end{equation*}
Since $3<q<5$, there exists
$\mathbf{e}:=(t_{\beta}\phi,t_{\beta}\phi)$ satisfying that
$\Psi(\mathbf{e})<0$  and $\|\mathbf{e}\|>\delta$ for some
$t_{\beta}$ is large enough.
\end{proof}

 We put
 \begin{equation*}
    c_{*}(\beta)\stackrel{\mathrm{def}}{=}\inf_{\gamma\in\Gamma_{\beta}}\sup_{t\in[0,1]}\Psi(\gamma(t))
 \end{equation*}
 where
 $$\Gamma_{\beta}:=\{\gamma\in\mathcal{C}([0,1],\mathcal{H})|\quad
 \gamma(0)=0,\quad  \gamma(1)=\mathbf{e}
             \}.$$
\begin{lem}\label{concentration lemma}
Under the assumptions of Theorem \ref{thm3},
$c_{*}(\beta)\to 0$ as $\beta\to \infty$.
\end{lem}
\begin{proof}
By Lemma \ref{Mountainpass Geometry}, there exists $t_{\beta}>0$
such that $\Psi(t_{\beta}\phi,t_{\beta}\phi)=\max_{t\ge
0}\Psi(t\phi,t\phi)$. Hence
\begin{equation}\label{last lemma}
(a_1+a_2)\|\phi\|^2
    t_{\beta}+(b_1+b_2)\|\phi\|^4 t^3_{\beta}
    =(\lambda_1+\lambda_2)|\phi|^6_6 t^5_{\beta}+
2\beta|\phi|^{q+1}_{q+1} t^{q}_{\beta},
\end{equation}
which implies that $t_{\beta}$  and $\beta t_{\beta}$ are uniformly
bounded w.r.t. $\beta$.

Going if necessary to a subsequence, we can assume that
$t_{\beta}\to 0$ as $\beta\to \infty$. Define
$\gamma_{\beta}(t):=t\mathbf{e}(\beta)$ for any $t\in [0,1]$, then
$\gamma_{\beta}\in\Gamma_{\beta}$.

This jointly with $(\ref{last lemma})$ yields that
\begin{eqnarray*}
  0\le c_{*}(\beta) &\le& \max_{t\ge 0}\Psi(\gamma_{\beta}(t))=\Psi(t_{\beta}\phi,t_{\beta}\phi) \\
   &\le& \frac{a_1+a_2}{3}\|\phi\|^2
    t_{\beta}^2+\frac{b_1+b_2}{12}\|\phi\|^4 t^4_{\beta}.
\end{eqnarray*}
Therefore $c_{*}(\beta)\to 0$ as $\beta\to \infty$.
\end{proof}
By the Mountain Pass Lemma without $(PS)$ condition (see
\cite{Ambrosetti-Rabinowitz73} or \cite{Willem}), there exists a
 sequence
$\{\mathbf{u}_n:=(u_{n,1},u_{n,2})\}\subset\mathcal{H}$ such that
    $\Psi(\mathbf{u}_n)\to c_{*}$ and $\Psi^{'}(\mathbf{u}_n)\to 0$.
We assume that $(u_{n,1},u_{n,2})$ is nonnegative, otherwise we
consider $(|u_{n,1}|,|u_{n,2}|)$.

 For $n$ large enough, we have that
 \begin{eqnarray*}
   c_{*}(\beta)+\|u_{n,1}\|+\|u_{n,2}\|&\ge& \Psi(u_{n,1},u_{n,2})-\frac{1}{q+1}\langle\Psi^{'}(u_{n,1},u_{n,2}),(u_{n,1},u_{n,2})\rangle\\
   &=& \frac{q-1}{2(q+1)}(a_1\|u_{n,1}\|^2+a_2\|u_{n,2}\|^2)+\frac{q-3}{4(q+1)}(b_1\|u_{n,1}\|^4+b_2\|u_{n,2}\|^4)\\
   && -\frac{5-q}{6(q+1)}(\lambda_1|u_{n,1}^+|^6+\lambda_2|u_{n,2}^+|^6)\\
    &\ge& C(\|u_{n,1}\|^2+\|u_{n,2}\|^2).
 \end{eqnarray*}
    It follows that $\|u_{n,1}\|$ and $\|u_{n,2}\|$ are bounded.
Going if necessary to a subsequence, we can assume that
$u_{n,1}\rightharpoonup \bar{u}_1$ and $u_{n,2}\rightharpoonup
\bar{u}_2$.

\begin{lem}\label{4.3} For $\beta>0$  large enough, $u_{n,1}\to\bar{u}_1$ and $u_{n,2}\to\bar{u}_2$ in $H$ as $n\to
\infty$; Moreover, $(\bar{u}_1,\bar{u}_2)\not\equiv(0,0)$.
\end{lem}
\begin{proof}
The main idea of this proof is motivated by Alves-Correa-Figueiredo
\cite{Alves-Correa-Fig2010} in the case of a single equation.
Suppose that $|\nabla u_{n,1}|^2\rightharpoonup|\nabla
\bar{u}_1|^2+\mu$
 and $|u_{n,1}^6|\rightharpoonup |\bar{u}_1|^6+\nu$. The concentration
 compactness principle due to Lions
 yields that there exist an at most countable index set $\Lambda$,
 sequences $\{\mathbf{x}_i\}\subset\mathbb{R}^3$,
 $\{\mu_i\}$,$\{\nu_i\}\subset[0,\infty)$,
 such that for any $i\in\Lambda$,
 \begin{equation}\label{concentration1}
 \nu=\sum_{i\in\Lambda}\nu_i\delta_{\mathbf{x}_i},\quad
 \mu\ge\sum_{i\in\Lambda}\mu_i\delta_{\mathbf{x}_i}\qquad
 S\nu_i^{\frac{1}{3}}\le\mu_i.
 \end{equation}
Define a cut-off function by
$\varphi_{\rho}(x):=\varphi(\frac{\mathbf{x}-\mathbf{x}_i}{\rho})$
for every $\rho>0$ where
$\varphi\in\mathcal{C}^{\infty}_0(\Omega,[0,1])$ satisfying
$\varphi\equiv 1$ on $B_1(0)$, $\varphi\equiv 0$ on $\Omega\setminus
B_2(0)$
 and $|\nabla\varphi|_{\infty}\le 2$.

  Hence, $$\langle\Psi^{'}(u_{n,1},u_{n,2}),(\varphi_{\rho}u_{n,1},0)\rangle\to
 0,$$ that is,
\begin{eqnarray}
\nonumber
   &&(a_1+b_1\|u_{n,1}\|^2)\int_{\Omega}u_{n,1}\nabla u_{n,1}\nabla
   \varphi_{\rho}\, \mathrm{d}x
   +(a_1+b_1\|u_{n,1}\|^2)\int_{\Omega}\varphi_{\rho}|\nabla u_{n,1}|^2\,\mathrm{d} x \\
   &=&\lambda_1\int_{\Omega}\varphi_{\rho}|u_{n,1}|^6\,\mathrm{d}x
   +\beta\int_{\Omega}\varphi_{\rho}|u_{n,1}|^{\frac{q+1}{2}}|u_{n,2}|^{\frac{q+1}{2}}\,\mathrm{d}x
   +o(1)\label{concentration2}
\end{eqnarray}

By  direct calculations  and H{\"o}lder inequality, we have that
\begin{eqnarray*}
    && \lim_{\rho\to 0}\big[\lim_{n\to\infty}(a_1+b_1\|u_{n,1}\|^2)\int_{\Omega}u_{n,1}\nabla u_{n,1}\nabla
   \varphi_{\rho}\, \mathrm{d}x\big] \\
   &\le& \lim_{\rho\to 0}\big[\lim_{n\to\infty}(a_1+b_1\|u_{n,1}\|^2)|u_{n,1}|_3\|u_{n,1}\|(\int_{\Omega}|\nabla
   \varphi_{\rho}|^6\, \mathrm{d}x)^{\frac{1}{6}}\big] \\
   &\le& C\lim_{\rho\to 0}\rho^{2}=0,\\
  && \lim_{\rho\to 0}\big[\lim_{n\to\infty}(a_1+b_1\|u_{n,1}\|^2)\int_{\Omega}\varphi_{\rho}|\nabla u_{n,1}|^2\, \mathrm{d}x\big] \\
  &\ge& \lim_{\rho\to 0}a_1\int_{\Omega}\varphi_{\rho}(|\nabla \bar{u}_1|^2+\mu)\, \mathrm{d}x \\
  &=&\lim_{\rho\to 0}a_1\int_{\Omega}\varphi_{\rho}\mu\,
  \mathrm{d}x\\
  &\ge& a_1\mu_i,\\
   && \lim_{\rho\to 0}\lim_{n\to\infty}\lambda_1\int_{\Omega}\varphi_{\rho}|u_{n,1}|^6\,\mathrm{d}x\\
  &=& \lim_{\rho\to 0}\lambda_1\int_{\Omega}\varphi_{\rho}(|\bar{u}_1|^6+\nu)\,\mathrm{d}x\\
  &=&\lambda_1\nu_i,\\
   &&\lim_{\rho\to 0}\lim_{n\to\infty}\beta\int_{\Omega}\varphi_{\rho}|u_{n,1}|^{\frac{q+1}{2}}|u_{n,2}|^{\frac{q+1}{2}}\,\mathrm{d}x  \\
  &\le&\lim_{\rho\to
  0}\lim_{n\to\infty}\beta\left(\int_{\Omega}\varphi_{\rho}^{\frac{12}{5-q}}\,\mathrm{d}x\right)^{\frac{5-q}{12}}\left(\int_{\Omega}|u_{n,1}|^6\,\mathrm{d}x\right)^{\frac{q+1}{12}}\left(\int_{\Omega}|u_{n,2}|^{q+1}\,\mathrm{d}x\right)^{\frac{1}{2}}\\
  &\le& C\lim_{\rho\to 0}\rho^{3}=0.
\end{eqnarray*}
Therefore, letting $n\to\infty$ and $\rho\to 0$ in
(\ref{concentration2}), we get that $\lambda_1\nu_i\ge a_1\mu_i$.
This jointly (\ref{concentration1}), yields that
$\nu_i\ge(a_1S/\lambda_1)^{\frac{3}{2}}$.

Since \begin{eqnarray*}
        c_{*}(\beta) &=& \Psi(u_{n,1},u_{n,2})-\frac{1}{q+1}\langle\Psi^{'}(u_{n,1},u_{n,2}),(u_{n,1},u_{n,2})\rangle +o(1)\\
        &\ge & \frac{5-q}{6(q+1)}\int_{\Omega}\lambda_1|u_{n,1}|^6\,\mathrm{d}x+o(1)\\
        &\ge &
        \frac{5-q}{6(q+1)}\int_{\Omega}\lambda_1\varphi_{\rho}|u_{n,1}|^6\,\mathrm{d}x+o(1),
      \end{eqnarray*}
letting $n\to\infty$, we obtain that
\begin{eqnarray*}
  c_{*}(\beta) &\ge& \frac{5-q}{6(q+1)}\lambda_1\sum_{i\in\Lambda}\varphi_{\rho}(\mathbf{x}_i)\nu_i  \\
     &=& \frac{5-q}{6(q+1)}\lambda_1\sum_{i\in\Lambda}\nu_i \\
     &\ge& \frac{5-q}{6(q+1)}\lambda_1^{-\frac{1}{2}}(a_1S)^{\frac{3}{2}}>0.
\end{eqnarray*}
This contradicts with Lemma \ref{concentration lemma} when $\beta$
is large enough. Hence, $\Lambda$ is empty and it follows that
$u_{n,1}\to \bar{u}_1$ in $L^6(\Omega)$. In the same way, we get
that $u_{n,2}\to \bar{u}_2$ in $L^6(\Omega)$. By similar arguments
as in Lemma \ref{2.1}$(v)$, we can prove that $u_{n,1}\to\bar{u}_1$
and $u_{n,2}\to \bar{u}_2$ in $H$.

Moreover, we have that for fixed $\beta>0$,
$$\Psi(\bar{u}_1,\bar{u}_2)=c_{*}(\beta)\ge\alpha>0,$$
with $\alpha$ obtained in \emph{Lemma \ref{Mountainpass Geometry} (i)}.
Hence, $(\bar{u}_1,\bar{u}_2)\not\equiv(0,0)$.
\end{proof}

{\textbf{\emph Proof of Theorem 1.3}}

Lemma \ref{4.3} immediately yields that $(\bar{u}_1,\bar{u}_2)$ is a
nonnegative solution of problem (\ref{critical system}) as $\beta$
is large enough.

In order to prove Theorem \ref{thm3}, it is enough for us to prove
$\bar{u}_1>0$ and $\bar{u}_2>0$.

In fact, if $\bar{u}_2\equiv 0$, then
$\bar{u}_1\ge 0$ and $\bar{u}_1\not\equiv 0$ which is   a  solution
of
\begin{equation}\label{the fct of J}
    \begin{cases}
    -(a_1+b_1\|u\|^2)\Delta u =\lambda_1 (u^{+})^5\qquad in \quad\Omega,\\
         u=0\quad on \quad\partial{\Omega}.
\end{cases}
\end{equation}
Then, maximum principle implies that  $\bar{u}_1>0$. Moreover,
$\bar{u}_1$ is a critical point of
$$J(u)=\frac{1}{2}a_1\|u\|^2+\frac{1}{4}\|u\|^4+\lambda_1|u^{+}|_6^6, \qquad \forall u\in H.$$

We can define the Nehari manifolds corresponding to $\Psi$ and $J$
respectively by
\begin{equation*}
    \mathcal{N}=\left\{\;\mathbf{u}\in\mathcal{H}\setminus\{\mathbf{0}\}\left|
\begin{array}{cc}&a_1\|u_1\|^2+a_2\|u_2\|^2+
    b_1\|u_1\|^4+b_2\|u_2\|^4\\
  =& \int_{\Omega}(\lambda_1|u_1^{+}|^{6}+\lambda_2|u_2^{+}|^{6}
    +2\beta|u_1^{+}|^{\frac{q+1}{2}}|u_2^{+}|^{\frac{q+1}{2}})\,\mathrm{d}x
    \end{array}
    \right.
    \right\},
\end{equation*}
 and
 \begin{equation*}
    \tilde{\mathcal{N}}=\{u\in H\setminus\{0\}|
    a_1\|u\|^2+b_1\|u\|^4=\lambda_1|u^{+}|_6^6\}.
 \end{equation*}

We claim that $\mathcal{N}$ is homeomorphic to the unit sphere
of $\mathcal{H}$. Indeed, for fixed $\mathbf{u}\in\mathcal{H}\setminus\{\mathbf{0}\}$, $t\mathbf{u}\in\mathcal{N}$  if and only if $t$ is a positive zero of the following function on $\mathbb{R}$
\begin{equation*}
    g(t):= At^4+Bt^{q-1}-Ct^2-D,
\end{equation*}
where $A,B,C$ and $D$ are positive constants (depending on
$\mathbf{u}$). In order to examine the zeros of $g$, we firstly
consider its derivative $g^{'}(t)=t[4At^2+(q-1)Bt^{q-3}-2C]$, and
define $h(t)=4At^2+(q-1)Bt^{q-3}-2C$. Observing that
$h(0)<0,\,h(+\infty)=+\infty$ and $h^{'}(t)>0$ $\forall$
$t\in\mathbb{R}^{+}$, we can conclude that  there exists a unique
$\bar{t}>0$ such that $h^{'}(\bar{t})=0$. In other words, there
exists a unique $\bar{t}>0$ such that $g^{'}(\bar{t})=0$. Notice
that $g(0)<0$, $g^{'}(0)=0$, $g^{''}(0)<0$ and $g(+\infty)=+\infty$.
Thus, $g(t)$ has a unique positive zero, and hence, for any
$\mathbf{u}\in\mathcal{H}\setminus\{\mathbf{0}\}$, there exists a
unique $t(\mathbf{u})\in \mathbb{R}^{+}$ such that
$t(\mathbf{u})\mathbf{u}\in\mathcal{N}$. By  similar arguments as in
the proof of $Lemma$ \ref{2.1} $(i)$, we have that $\mathcal{N}$ is
homeomorphic to the unit sphere of $\mathcal{H}$.

It follows from similar arguments as in $Theorem$ $4.2$ \cite{Willem} that
$$c_{*}(\beta)=\inf_{\mathbf{u}\in\mathcal{N}}\Psi(\mathbf{u}).$$
Furthermore,
 \begin{equation*}
    J(\bar{u}_1)=\Psi(\bar{u}_1,0)=\inf_{\mathbf{u}\in
    \mathcal{N}}\Psi(\mathbf{u})=\inf_{\mathbf{u}\in\tilde{\mathcal{N}}\times\{0\}}\Psi(\mathbf{u})=\inf_{u\in\tilde{\mathcal{N}}}
    J(u).
 \end{equation*}
Thus, $\bar{u}_1$ is a least energy solution of problem (\ref{the
fct of J}).

 It follows that
\begin{equation*}
    \inf_{u\in H\setminus\{0\} \atop |u|_6=1}a_1\|u\|^2+b_1\|u\|^4
\end{equation*}
can be attained on $\Omega\not=\mathbb{R}^3$, which contradicts with
the fact that

{\emph{S is never attained on domains
$\Omega\subset\mathbb{R}^3$,$\Omega\not=\mathbb{R}^3$}}(Theorem
$III.1.2,$\cite{Struwe}).

Hence, $\bar{u}_2\not\equiv 0$. Similar arguments implies that
$\bar{u}_1\not\equiv 0$. The strong maximum principle can be applied
to each equation of problem (\ref{critical system}) and implies that
$\bar{u}_1>0$ and $\bar{u}_2>0$. This completes the proof of Theorem
\ref{thm3}. \hspace{\stretch{1}}$\Box$

  {\textbf{\emph Proof of Corollary \ref{cor 1.4}}}

Let us recall that for all $\mathbf{u}\in\mathcal{H}$
\begin{eqnarray*}
  {\Phi}^{+}(\mathbf{u}) &=& \frac{1}{2}(a_1\|u_1\|^2+a_2\|u_2\|^2)+
    \frac{1}{4}(b_1\|u_1\|^4+b_2\|u_2\|^4) \\
  &&-\frac{1}{p+1}\int_{\Omega}\left(\lambda_1|u_1^+|^{p+1}+\lambda_2|u_2^+|^{p+1}
    +2\beta|u_1^+|^{\frac{p+1}{2}}|u_2^+|^{\frac{p+1}{2}}\right)\,\mathrm{d}x.
\end{eqnarray*}
We define
\begin{equation*}
    c^{+}_{*}(\beta)\stackrel{\mathrm{def}}{=}\inf_{\gamma\in\Gamma^{+}_{\beta}}\sup_{t\in[0,1]}\Phi^{+}(\gamma(t)),
 \end{equation*}
 where
 $$\Gamma^{+}_{\beta}:=\{\gamma\in\mathcal{C}([0,1],\mathcal{H})|\quad \gamma(0)=0,\quad
             \Phi^{+}(\gamma(1))<0\}.$$
             It follows from condition (\ref{subcritical case}) that
             $\Phi^{+}$ satisfies the $(PS)$ condition.

 Using similar arguments as in \emph{Lemma \ref{Mountainpass Geometry}}
 and \emph{Lemma \ref{concentration lemma}}, we have:
\begin{enumerate}
  \item  there exist positive numbers $\delta$ and $\alpha$ (depending on
$\beta$) such that for any $\mathbf{u}\in\mathcal{H}$ with
$\|\mathbf{u}\|=\delta$, $\Phi^{+}(\mathbf{u})\ge\alpha>0$.
  \item  for  $\beta>0$, there exists
$\mathbf{e}=\mathbf{e}(\beta)\in\mathcal{H}$ with
$\Phi^{+}(\mathbf{e})<0$ and $\|\mathbf{e}\|\ge\delta$.
  \item  under the assumptions of Theorem \ref{thm1},
$c^{+}_{*}(\beta)\to 0$ as $\beta\to \infty$.
\end{enumerate}

By the Mountain Pass Lemma with $(PS)$ condition (see
\cite{Ambrosetti-Rabinowitz73} or \cite{Willem}), we obtain a
critical point $\tilde{\mathbf{u}}=(\tilde{u}_1,\tilde{u}_2)$ of
$\Phi^{+}$ satisfying that
$$\Phi^{+}(\tilde{\mathbf{u}})=c^{+}_{*}(\beta),\qquad  \tilde{\mathbf{u}}\ge 0\qquad and \quad \tilde{\mathbf{u}}\not\equiv \mathbf{0}.$$
If $\tilde{u}_2\equiv 0$, then $\tilde{u}_1$ is a positive solution
of problem (\ref{2}) with $a=a_1$, $b=b_1$ and
$f(x,u)=\lambda_1(u^{+})^{p}$. The condition that $N=1$ or $\Omega$
is a radially symmetric ensures that such a problem has a uniqueness
positive solution (see \emph{Remark \ref{rem 1.1}}). Hence
$\tilde{u}_1=U_1$. It follows that
\begin{equation*}
    \frac{p-1}{2(p+1)}a_1\|U_1\|^2+\frac{p-3}{4(p+1)}b_1\|U_1\|^4=\Phi^{+}(U_1,0)=c^{+}_{*}(\beta)\to \,0,\qquad as\quad\beta\to \, +\infty,
\end{equation*}
which is a contradiction. Thus, we have proved that $\tilde{u}_2\not\equiv 0$.
Similar arguments implies  $\tilde{u}_1\not\equiv 0$. The strong
maximum principle can be applied to each equation of problem
(\ref{1}) and yields that $\tilde{u}_1>0$ and $\tilde{u}_2>0$.

Now, we have obtained a positive least energy solution
$\tilde{\mathbf{u}}$ of problem $(\ref{1})$ when $\beta>0$ is
sufficiently large. Moreover,
\begin{equation*}
    \Phi^{+}(\mathbf{u}^{*})\ge
    \,max\{\Phi^{+}(\mathbf{u}_1),\Phi^{+}(\mathbf{u}_2)\} >\Phi^{+}(\tilde{\mathbf{u}})=c^{+}_{*}(\beta)\to
    0, as \,\beta\to+\infty,
\end{equation*}
where $\mathbf{u}^{*}$ is the solution obtained in Theorem
$\ref{thm1}$. Thus, we get two positive solutions of problem
(\ref{1}) for sufficiently large $\beta>0$.
 \hspace{\stretch{1}}$\Box$

\bibliographystyle{plain}
  \makeatletter
   \def\@biblabel#1{#1.~}
\makeatother


\providecommand{\bysame}{\leavevmode\hbox to3em{\hrulefill}\thinspace}
\providecommand{\MR}{\relax\ifhmode\unskip\space\fi MR }
\providecommand{\MRhref}[2]{%
  \href{http://www.ams.org/mathscinet-getitem?mr=#1}{#2}
}
\providecommand{\href}[2]{#2}

\end{document}